\documentclass[11pt]{amsart}

\usepackage{latexsym}
\usepackage{amssymb}
\usepackage{amsmath}

\newtheorem{theorem}{Theorem}

\theoremstyle{definition}

\theoremstyle{remark}

\def\N{\mathbb{N}}
\def\Z{\mathbb{Z}}
\def\Q{\mathbb{Q}}
\def\R{\mathbb{R}}

\def\hN{{{}^*\N}}
\def\hZ{{{}^*\Z}}
\def\hQ{{{}^*\Q}}
\def\hR{{{}^*\R}}

\def\hhN{{}^{**}\N}

\def\hA{{{}^*\!A}}
\def\hB{{{}^*\!B}}
\def\hC{{{}^*\!C}}
\def\hX{{{}^*X}}
\def\hY{{{}^*Y}}

\def\hf{{{}^*\!f}}

\def\F{\mathcal{F}}
\def\U{\mathcal{U}}
\def\V{\mathcal{V}}

\def\UU{\mathfrak{U}}
\newcommand{\fe}{\ensuremath{\leq_{\text{fe}}}}

\begin{document}

\title[Nonstandard methods in combinatorics of numbers]
{A taste of nonstandard methods
in combinatorics of numbers}

\author{Mauro Di Nasso}

\address{Dipartimento di Matematica\\
Universit\`a di Pisa, Italy}

\email{dinasso@dm.unipi.it}

\date{}

\begin{abstract}
\emph{By presenting the proofs of a few sample results,
we introduce the reader to the use of nonstandard analysis
in aspects of combinatorics of numbers.}
\end{abstract}



\maketitle

\section*{Introduction}

In the last years, several combinatorial results about sets of
integers that depend on their asymptotic densities have been
proved by using the techniques of nonstandard analysis,
starting from the pioneering work by R. Jin
(see \emph{e.g.} \cite{jin0,jin1,jin2,jin3,jin4,dn1,dgjllm1,dgjllm2}).
Very recently, the hyper-integers of nonstandard analysis have also been
used in Ramsey theory to investigate the partition regularity of
possibly non-linear diophantine equations (see \cite{dn1,lu}).

\smallskip
The goal of this paper is to give a soft introduction to
the use of nonstandard methods in certain areas of density
problems and Ramsey theory. To this end, we will focus on a
few sample results, aiming to give the
flavor of how and why nonstandard techniques could be
successfully used in this area.

\smallskip
Grounding on nonstandard definitions of the involved
notions, the presented proofs consist of arguments
that can be easily followed by the intuition
and that can be taken at first as heuristic reasonings.
Subsequently, in the last foundational section, we will outline
an algebraic construction of the hyper-integers,
and give hints to show how those nonstandard arguments are in fact
rigorous ones when formulated in the appropriate language.
We will also prove that all the nonstandard definitions
presented in this paper are actually
equivalent to the usual ``standard" ones.

\smallskip
Two disclaimers are in order. Firstly, this paper
is not to be taken as a comprehensive presentation
of nonstandard methods in combinatorics, but only
as a taste of that area of research.
Secondly, the presented results are only examples of
``first-level" applications of the nonstandard machinery;
for more advanced results one needs higher-level nonstandard tools,
such as saturation and Loeb measure, combined with
other non-elementary mathematical arguments.

\bigskip
\section{The hyper-numbers of nonstandard analysis}

This introductory section contains an informal description of the basics of
nonstandard analysis, starting with the hyper-natural numbers.
Let us stress that what follows are not rigorous definitions and results,
but only informal discussions aimed to help the intuition
and provide the essential tools to understand the rest of the paper.\footnote
{~A model for the introduced notions will be constructed
in the last section.}

\smallskip
One possible way to describe the hyper-natural numbers
$\hN$ is the following:

\medskip
\begin{itemize}
\item
The \emph{hyper-natural numbers} $\hN$ are the natural numbers when
seen with a ``telescope" which allows to also see
infinite numbers beyond the usual finite ones.
The structure of $\hN$ is essentially the same as $\N$,
in the sense that $\hN$ and $\N$ cannot be distinguished
by any ``elementary property".
\end{itemize}

\medskip
Here by \emph{elementary property} we mean a property that
talks about elements but \emph{not} about subsets\footnote
{~In logic, this kind of properties are
called \emph{first-order} properties.},
and where no use of the notion of infinite or finite
number is made.

\smallskip
In consequence of the above, the order structure of $\hN$ is clear.
After the usual finite numbers $\N=\{1,2,3,\ldots\}$,
one finds the infinite numbers $\xi>n$ for all $n\in\N$.
Every $\xi\in\hN$ has a successor $\xi+1$,
and every non-zero $\xi\in\hN$ has a predecessor $\xi-1$.
$${}^*\N\ =\ \big\{\underbrace{1,2,3,\ldots,n,\ldots}_{\text{\small finite numbers}}
\quad \underbrace{\ldots, N-2, N-1, N, N+1, N+2, \ldots}_{\text{\small infinite numbers}}\
\big\}$$

\smallskip
Thus the set of finite numbers $\N$ has not a
greatest element and the set of infinite numbers
$\N_\infty=\hN\setminus\N$ has not a least element,
and hence $\hN$ is \emph{not} well-ordered.
Remark that being a well-ordered set is not an ``elementary
property" because it is about subsets, not elements.\footnote
{~In logic, this kind of properties are
called \emph{second-order} properties.}

\medskip
\begin{itemize}
\item
The \emph{hyper-integers} $\hZ$ are the discretely ordered
ring whose positive part is the semiring $\hN$.

\smallskip
\item
The \emph{hyper-rationals} $\hQ$ are the ordered
field of fractions of $\hZ$.
\end{itemize}

\medskip
Thus $\hZ=-\hN\cup\{0\}\cup\hN$, where $-\hN=\{-\xi\mid\xi\in\hN\}$
are the negative hyper-integers.
The hyper-rational numbers $\zeta\in\hQ$ can be represented
as ratios $\zeta=\frac{\xi}{\nu}$ where $\xi\in\hZ$ and $\nu\in\hN$.

\smallskip
As the next step, one considers the hyper-real numbers,
which are instrumental in nonstandard calculus.

\medskip
\begin{itemize}
\item
The \emph{hyper-reals} $\hR$ are an ordered field
that properly extends both $\hQ$ and $\R$.
The structures $\R$ and $\hR$ satisfy the same
``elementary properties".
\end{itemize}

\medskip
As a proper extension of $\R$, the field $\hR$
is \emph{not} Archimedean, \emph{i.e.} it contains
non-zero \emph{infinitesimal} and \emph{infinite} numbers.
(Recall that a number $\varepsilon$ is infinitesimal if
$-1/n<\varepsilon<1/n$ for all $n\in\N$; and a number
$\Omega$ is infinite if $|\Omega|>n$ for all $n$.)
In consequence, the field $\hR$ is \emph{not} complete:
\emph{e.g.}, the bounded set of
infinitesimals has not a least upper bound.\footnote
{~Remark that the property of completeness is \emph{not} elementary,
because it talks about subsets and not about elements of the
given field. Also the Archimedean property is \emph{not} elementary,
because it requires the notion of \emph{finite} hyper-natural number
to be formulated.}

\smallskip
Each set $A\subseteq\R$ has its
\emph{hyper-extension} $\hA\subseteq\hR$, where $A\subseteq\hA$.
\emph{E.g.}, one has the set of hyper-even numbers,
the set of hyper-prime numbers, the set of hyper-irrational numbers,
and so forth. Similarly, any function $f:A\to B$ has
its \emph{hyper-extension} $\hf:\hA\to\hB$, where
$\hf(a)=f(a)$ for all $a\in A$.
More generally, in nonstandard analysis one considers
hyper-extensions of arbitrary sets and functions.

\smallskip
The general principle that hyper-extensions are indistinguishable
from the starting objects as far as their ``elementary properties" are
concerned, is called \emph{transfer principle}.

\medskip
\begin{itemize}
\item
\emph{Transfer principle}:\
An ``elementary property" $P$ holds for the
sets $A_1,\ldots,A_k$ and the
functions $f_1,\ldots,f_h$ if and only if $P$ holds
for the corresponding hyper-extensions:
$$\quad\quad\quad\ \
P(A_1,\ldots,A_k,f_1,\ldots,f_h)\ \Longleftrightarrow\
P(\hA_1,\ldots,\hA_k,{}^*\!f_1,\ldots,{}^*\!f_h)$$
\end{itemize}

\medskip
Remark that all basic set properties are elementary,
and so $A\subseteq B\Leftrightarrow \hA\subseteq\hB$,
$A\cup B=C\Leftrightarrow\hA\cup\hB=\hC$,
$A\setminus B=C\Leftrightarrow \hA\setminus\hB=\hC$, 
and so forth.

\smallskip
As direct applications of \emph{transfer} one
obtains the following facts:
The hyper-rationals $\hQ$ are \emph{dense} in the hyper-reals $\hR$;
every hyper-real number $\xi\in\hR$ has an an \emph{integer part}, \emph{i.e.}
there exists a unique hyper-integer $\mu\in\hZ$ such that
$\mu\le\xi<\mu+1$; and so forth.

\smallskip
As our first example of nonstandard reasoning, let us see
a proof of K\"onig's Lemma, one of the
oldest results in infinite combinatorics.

\medskip
\begin{theorem}[K\"onig's Lemma -- 1927]
If a finite branching tree has infinitely many
nodes, then it has an infinite branch.
\end{theorem}

\begin{proof}[Nonstandard proof]
Given a finite branching tree $T$, consider the sequence of
its finite levels $\langle T_n\mid n\in\N\rangle$, and
let $\langle T_\nu\mid \nu\in\hN\rangle$ be its hyper-extension.
By the hypotheses, it follows that
all finite levels $T_n\ne\emptyset$ are nonempty. Then, by \emph{transfer},
also all ``hyper-levels" $T_\nu$ are nonempty.
Pick a node $\tau\in T_\nu$ for some infinite $\nu$.
Then $\{t\in T\mid t\le \tau\}$ is an infinite branch of $T$.
\end{proof}

\bigskip
\section{Piecewise syndetic sets}

A notion of largeness used in combinatorics
of numbers is the following.

\medskip
\begin{itemize}
\item
A set of integers $A$ is \emph{thick} if it includes arbitrarily
long intervals:
$$\forall n\in\N\ \exists x\in\Z\ \ [x,x+n)\subseteq A.$$
\end{itemize}

\medskip
In the language of nonstandard analysis:

\medskip
\noindent
\textbf{Definition.}\ [Nonstandard]\
$A$ is \emph{thick} if $I\subseteq\hA$
for some infinite interval $I$.

\medskip
By \emph{infinite interval} we mean an interval
$[\nu,\mu]=\{\xi\in\hZ\mid \nu\le\xi\le\mu\}$
with infinitely many elements or, equivalently,
an interval whose length $\mu-\nu+1$ is an infinite number.

\smallskip
Another important notion is that of syndeticity. It stemmed from dynamics,
corresponding to finite return-time in a discrete setting.

\medskip
\begin{itemize}
\item
A set of integers $A$ is \emph{syndetic} if it has bounded gaps:
$$\exists k\in\N\ \forall x\in\Z\ \ [x,x+k)\cap A\ne\emptyset.$$
\end{itemize}

\medskip
So, a set is syndetic means that its complement is not thick.
In the language of nonstandard analysis:

\medskip
\noindent
\textbf{Definition.}\ [Nonstandard]\
$A$ is \emph{syndetic} if $\hA\cap I\neq\emptyset$
for every infinite interval $I$.

\medskip
The fundamental structural property considered in Ramsey theory
is that of partition regularity.

\medskip
\begin{itemize}
\item
A family $\F$ of sets is \emph{partition regular} if
whenever an element $A\in\F$ is finitely
partitioned $A=A_1\cup\ldots\cup A_n$, then at least one piece $A_i\in\F$.
\end{itemize}

\medskip
Remark that the family of syndetic sets fails to be partition regular.\footnote
{~\emph{E.g.}, consider the partition of the integers determined by
$$A=\bigcup_{n\in\N}[-2^{2n},-2^{2n-1})\cup\bigcup_{n\in\N}[2^{2n-1},2^{2n})$$
and its complement $\Z\setminus A$, neither of which is syndetic.}
However, a suitable weaking of syndeticity satisfies the property.

\medskip
\begin{itemize}
\item
A set of integers $A$ is \emph{piecewise syndetic} if
$A=T\cap S$ where $T$ is thick and $S$ is syndetic;
\emph{i.e.}, $A$ has bounded gaps on arbitrarily large intervals:
$$\exists k\in\N\ \forall n\in\N\ \exists y\in\Z\ \forall x\in\Z\ \
[x,x+k)\subseteq[y,y+n)\Rightarrow [x,x+k)\cap A\ne\emptyset.$$
\end{itemize}

\medskip
In the language of nonstandard analysis:

\medskip
\noindent
\textbf{Definition.}\ [Nonstandard]\
$A$ is \emph{piecewise syndetic} (PS for short) if there exists an infinite
interval $I$ such that $\hA\cap I$ has only finite gaps,
\emph{i.e.} $\hA\cap J\ne\emptyset$ for every infinite subinterval $J\subseteq I$.

\medskip
Several results suggest the notion of piecewise
syndeticity as a relevant one in combinatorics of numbers. \emph{E.g.},
the sumset of two sets of natural numbers having positive
density is piecewise syndetic\footnote
{~This is \emph{Jin's theorem}, proved in 2000 by using
nonstandard analysis (see \cite{jin1}).};
every piecewise syndetic set contains arbitrarily long arithmetic progressions;
a set is piecewise syndetic if and only if it belongs
to a minimal idempotent ultrafilter\footnote{~See \cite{hs} \S 4.4.}.

\medskip
\begin{theorem}
The family of PS sets is partition regular.
\end{theorem}

\begin{proof}[Nonstandard proof]
By induction, it is enough to check the
property for $2$-partitions. So, let us assume
that $A=\text{BLUE}\cup\text{RED}$ is a PS set; we have
to show that $\text{RED}$ or $\text{BLUE}$ is PS.
We proceed as follows:

\smallskip
\begin{itemize}
\item
Take the hyper-extensions $\hA={}^*\text{BLUE}\cup{}^*\text{RED}$.

\smallskip
\item
By the hypothesis, we can pick an infinite
interval $I$ where $\hA$ has only finite gaps.

\smallskip
\item
If the ${}^*\text{blue}$ elements of $\hA$ have only
finite gaps in $I$, then $\text{BLUE}$ is piecewise syndetic.

\smallskip
\item
Otherwise, there exists an infinite interval $J\subseteq I$
that only contains ${}^*\text{red}$ elements of ${}^*A$.
But then ${}^*\text{RED}$ has only finite gaps in $J$, and
hence $\text{RED}$ is piecewise syndetic.
\end{itemize}

\vspace{-.55cm}
\end{proof}

\bigskip
\section{Banach and Shnirelmann densities}

An important area of research in number theory focuses on
combinatorial properties
of sets which depend on their density. Recall the following notions:

\medskip
\begin{itemize}
\item
\smallskip
The \emph{upper asymptotic density} $\overline{d}(A)$ of a set $A\subseteq\N$
is defined by putting:
$$\overline{d}(A)\ =\ \limsup_{n\to\infty}\frac{|A\cap[1,n]|}{n}\,.$$

\smallskip
\item
The \emph{upper Banach density} $\text{BD}(A)$ of a set
of integers $A\subseteq\Z$
generalizes the upper density by considering arbitrary intervals
in place of just initial intervals:
$$\text{BD}(A)\ =\
\lim_{n\to\infty}\left(\max_{x\in\Z}\frac{|A\cap[x+1,x+n]|}{n}\right)\ =\
\inf_{n\in\N}\left\{\max_{x\in\Z}\frac{|A\cap[x+1,x+n]|}{n}\right\}.$$
\end{itemize}

\medskip
In order to translate the above definitions in the language of nonstandard
analysis, we need to introduce new notions.

\smallskip
In addition to hyper-extensions, a larger class of well-behaved
subsets of $\hZ$ that is considered in nonstandard analysis
is the class of \emph{internal} sets. All sets that can be ``described"
without using the notions of finite or infinite number are internal.
Typical examples are the intervals
$$[\xi,\zeta]=\{x\in\hZ\mid \xi\le x\le\zeta\}\,;\ \
[\xi,+\infty)=\{x\in\hZ\mid \xi\ge x\}\,;\ \ \textit{etc}.$$

Also finite subsets $\{\xi_1,\ldots,\xi_n\}\subset\hZ$ are internal,
as they can be described by simply giving the (finite)
list of their elements. Internal subsets of $\hZ$
share the same ``elementary properties" of the subsets of $\Z$.
\emph{E.g.}, every nonempty internal subset of $\hZ$ that is bounded
below has a least element; in consequence, the set
$\N_\infty$ of infinite hyper-natural numbers is \emph{not} internal.
Internal sets are closed under unions, intersections, and relative
complements. So, also the set of finite numbers $\N$ is \emph{not} internal,
as otherwise $\N_\infty=\hN\setminus\N$ would be internal.

\smallskip
Internal sets are either \emph{hyper-infinite}
or \emph{hyper-finite}; for instance, all intervals $[\xi,+\infty)$
are hyper-infinite, and all intervals $[\xi,\zeta]$ are hyper-finite.
Every nonempty hyper-finite set $A\subset\hZ$ has its \emph{internal cardinality}
$\|A\|\in\hN$; for instance $\|[\xi,\zeta]\|=\zeta-\xi+1$.
Internal cardinality and the usual cardinality agree on finite sets.

\smallskip
If $\xi,\zeta\in\hR$ are hyperreal numbers, we write
$\xi\sim\zeta$ when $\xi$ and $\zeta$ are \emph{infinitely close},
\emph{i.e.} when their distance $|\xi-\zeta|$ is infinitesimal.
Remark that if $\xi\in\hR$ is finite (\emph{i.e.}, not infinite),
then there exists a unique real number $r\sim\xi$,
namely $r=\inf\{x\in\R\mid x>\xi\}$.\footnote
{~Such a real number $r$ is usually called the \emph{standard part} of $\xi$.}

\smallskip
We are finally ready to formulate the definitions of density in nonstandard terms.

\medskip
\noindent
\textbf{Definition.}\ [Nonstandard]\
For $A\subseteq\N$, its \emph{upper asymptotic density} $\overline{d}(A)=\beta$
is the greatest real number $\beta$ such that there exists
an infinite $\nu\in\hN$ with
$$\|\hA\cap [1,\nu]\|/\nu\ \sim\ \beta\,.$$

\smallskip
\noindent
\textbf{Definition.}\ [Nonstandard]\
For $A\subseteq\Z$, its \emph{upper Banach density}
$\text{BD}(A)=\beta$ is the greatest real number $\beta$
such that there exists an infinite interval $I$ with
$$\|\hA\cap I\|/\|I\|\ \sim\ \beta\,.$$

\medskip
Another notion of density that is
widely used in number theory is the following.

\medskip
\begin{itemize}
\item
The \emph{Schnirelmann density} $\sigma(A)$ of a set
$A\subseteq\N$ is defined by
$$\sigma(A)\ =\ \inf_{n\in\N}\frac{|A\cap[1,n]|}{n}.$$
\end{itemize}

\medskip
Clearly $\text{BD}(A)\ge\overline{d}(A)\ge\sigma(A)$, and it is easy
to find examples where inequalities are strict. Remark that
$\sigma(A)=1\Leftrightarrow A=\N$, and that
$\text{BD}(A)=1\Leftrightarrow A$ is thick.
Moreover, if $A$ is piecewise syndetic then
$\text{BD}(A)>0$, but not conversely.

\smallskip
Let us now recall a natural notion of embeddability
for the combinatorial structure of sets:\footnote
{~This notion is implicit in I.Z. Ruzsa's paper \cite{ru},
and has been explicitly considered in \cite{dn1} \S 4.
As natural as it is, it is well possible that finite
embeddability has been also considered by other authors,
but I am not aware of it.}

\medskip
\begin{itemize}
\item
We say that $X$ is \emph{finitely embeddable} in $Y$,
and write $X\fe Y$, if every finite $F\subseteq X$
has a shifted copy $t+F\subseteq Y$.
\end{itemize}

\medskip
It is readily seen that transitivity holds:
$X\fe Y$ and $Y\fe Z$ imply $X\fe Z$.
Notice that a set is $\fe$-maximal if and only if it is thick.
Finite embeddability preserves fundamental
combinatorial notions:

\medskip
\begin{itemize}
\item
If $X\fe Y$ and $X$ is PS,
then also $Y$ is PS.

\smallskip
\item
If $X\fe Y$ and $X$ contains an arithmetic progression of length $k$, then
also $Y$ contains an arithmetic progression of length $k$.

\smallskip
\item
If $X\fe Y$ then $\text{BD}(X)\le\text{BD}(Y)$.
\end{itemize}

\medskip
Remark that while piecewise syndeticity is preserved
under $\fe$, the property of being syndetic is \emph{not}. Similarly,
the upper Banach density is preserved or increased under $\fe$,
but upper asymptotic density is \emph{not}.

\smallskip
Other properties that suggest finite embeddability as a useful
notion are the following:

\medskip
\begin{itemize}
\item
If $X\fe Y$ then $X-X\subseteq Y-Y$\,;

\smallskip
\item
If $X\fe Y$ and $X'\fe Y'$ then $X-X'\fe Y-Y'$.
\end{itemize}

\medskip
In the nonstandard setting, $X\fe Y$ means that
a shifted copy of the whole $X$ is found in
the hyper-extension $\hY$.

\medskip
\noindent
\textbf{Definition.}\ [Nonstandard]\
$X\fe Y$ if $\nu+X\subseteq \hY$ for a suitable $\nu\in\hN$.

\medskip
Remark that the key point here is that
the shift $\nu$ could be an infinite number.

\smallskip
The sample result that we present below, due to R. Jin \cite{jin0},
allows to extend results that hold for sets with positive
Schnirelmann density to sets with positive upper Banach density.

\medskip
\begin{theorem}
Let $\text{BD}(A)=\beta>0$. Then there exists a set
$E\subseteq\N$ with $\sigma(E)\ge\beta$ and such that $E\fe A$.
\end{theorem}

\begin{proof}[Nonstandard proof]
By the nonstandard definition of Banach density,
there exists an infinite interval $I$ such that the
relative density $\|\hA\cap I\|/\|I\|\sim\beta$.
By translating if necessary, we can assume without loss of generality
that $I=[1,M]$ where $M\in\N_\infty$.
By a straight counting argument, we will prove the following:

\smallskip
\begin{itemize}
\item
\textbf{Claim.}\
\emph{For every $k\in\N$ there exists $\xi\in[1,M]$ such that
for all $i=1,\ldots,k$, the relative density
$\|\hA\cap[\xi,\xi+i)\|/i\ge\beta-1/k$.}
\end{itemize}

\smallskip
We then use an important principle of nonstandard analysis,
namely:

\smallskip
\begin{itemize}
\item
\emph{Overflow}:\
If $A\subseteq\hN$ is internal and contains all natural
numbers, then it also contains all hyper-natural
numbers up to an infinite $\nu$:
$$\quad\quad\quad\ \
A\ \text{internal}\ \&\ \N\subset A\ \Longrightarrow\
\exists\nu\in\N_\infty\ [1,\nu]\subseteq A.$$
\end{itemize}

\smallskip
By the Claim, the internal set below includes $\N$:
$$A\ =\ \{\nu\in\hN\mid\exists\xi\in[1,M]\ \forall i\le\nu\
\|\hA\cap[\xi,\xi+i)\|/i\ge\beta-1/\nu\}.$$
Then, by \emph{overflow},
there exists an infinite $\nu\in\hN$ and $\xi\in[1,M]$ such that
$\|\hA\cap[\xi,\xi+i)\|/i\ge\beta-1/\nu$ for all $i=1,\ldots,\nu$.
In particular, for all finite $n\in\N$, the real number
$\|\hA\cap[\xi,\xi+n)\|/n\ge\beta$ because it is not smaller than
$\beta-1/\nu$, which is infinitely close to $\beta$. If we denote by
$E=\{n\in\N\mid \xi+n\in\hA\}$, this means that $\sigma(E)\ge\beta$.
The thesis is reached because $\xi+E\subseteq\hA$,
and hence $E\fe A$, as desired.

\medskip
We are left to prove the Claim. Given $k$, assume by contradiction
that for every $\xi\in[1,M]$ there exists $i\le k$ such
that $\|\hA\cap[\xi,\xi+i)\|<i\cdot(\beta-1/k)$.
By ``hyper-induction" on $\hN$, define
$\xi_1=1$, and $\xi_{s+1}=\xi_s+n_s$ where
$n_s\le k$ is the least natural number such that
$\|\hA\cap[\xi_s,\xi_s+n_s)\|<n_s\cdot(\beta-1/k)$; and stop at step $N$
when $M-k\le\xi_N<M$. Since $k$ is finite, we have
$k/M\sim 0$ and ${\xi_N}/M\sim 1$. Then:
$$\beta\ \sim\ \frac{1}{M}\cdot\big\|\hA\cap[1,M]\big\|\ \sim\
\frac{1}{M}\cdot\big\|\hA\cap[\xi_1,\xi_N)\big\|\ =\
\frac{1}{M}\cdot\sum_{s=1}^{N-1}\big\|\hA\cap[\xi_s,\xi_{s+1})\big\|$$
$$<\ \frac{1}{M}\cdot\left(\sum_{s=1}^{N-1}n_s\cdot\left(\beta-\frac{1}{k}\right)\right)\ =\
\frac{\xi_N-1}{M}\cdot\left(\beta-\frac{1}{k}\right)\ \sim\ \beta-\frac{1}{k}\,,$$
a contradiction.
\end{proof}

\medskip
The previous theorem can be strengthened in several directions.
For instance, one can find $E$ to be ``densely" finitely
embedded in $A$, in the sense that
for every finite $F\subseteq X$ one has
``densely-many" shifted copies included in $Y$, \emph{i.e.}
$\text{BD}\left(\{t\in\Z\mid t+F\subseteq Y\}\right)>0$.\footnote
{~See \cite{dn1,dgjllm2} for more on this topic.}

\bigskip
\section{Partition regularity problems}

In this section we focus on the use
of hyper-natural numbers in partition regularity problems.

\smallskip
The notion of partition regularity for families of sets
given in Section 2, is sometimes weakened as follows:

\smallskip
\begin{itemize}
\item
A family $\F$ of sets is \emph{weakly partition regular}
on $X$ if for every finite partition $X=C_1\ldots\cup C_n$
there exists $F\in\F$ which is contained in one piece $F\subseteq C_i$.
\end{itemize}

\smallskip
Differently from the usual approach to nonstandard analysis,
here it turns out useful to work in a framework where
hyper-extensions can be iterated, so that one can
consider, \emph{e.g.}:

\medskip
\begin{itemize}
\item
The hyper-hyper-natural numbers $\hhN$\,;

\smallskip
\item
The hyper-extension ${}^*\xi\in\hhN$ of an
hyper-natural number $\xi\in\hN$\,;
\end{itemize}

\medskip
\noindent
and so forth. We remark that working with iterated hyper-extensions
requires caution, because of the existence of different
levels of extensions.\footnote
{~See \cite{dn2} for a discussion of the foundations
of iterated hyper-extensions.}
Here, it will be enough to
notice that, by \emph{transfer},
one has that $\hN\varsubsetneq\hhN$, and
if $\xi\in\hN\setminus\N$ then ${}^*\xi\in\hhN\setminus\hN$;
and similarly for $n$-th iterated hyper-extensions.\footnote
{~Notice also that $\hN$ is an initial segment of $\hhN$, \emph{i.e.}
$\xi<\nu$ for every $\xi\in\hN$ and for every
$\nu\in\hhN\setminus\hN$ (such a property is not used in this paper).}

\smallskip
Let us start with a nonstandard
proof of the classic Ramsey theorem for pairs.

\medskip
\begin{theorem}[Ramsey -- 1928]
Given a finite colouring $[\N]^2=C_1\cup\ldots\cup C_r$
of the pairs of natural numbers,
there exists an infinite set $H$ whose pairs are
monochromatic: $[H]^2\subseteq C_i$.\footnote
{~In other words, the family $\F=\{[H]^2\mid H\ \text{infinite}\}$
is weakly partition regular on $[\N]^2$.}
\end{theorem}

\begin{proof}[Nonstandard proof]
Take hyper-hyper-extensions and get the finite coloring
$$[{}^*\hN]^2\ =\ {}^{**}([\N]^2)\ =\ {}^{**}C_1\cup\ldots\cup{}^{**}C_r.$$
Pick an infinite $\xi\in{}^*\N$, let $i$ be such that
$\{\xi,{}^*\xi\}\in{}^{**}C_i$, and consider the set
$A=\{x\in\N\mid \{x,\xi\}\in{}^*C_i\}$. Then
$\xi\in\{x\in{}^*\N\mid\{x,{}^*\xi\}\in{}^{**}C_i\}=\hA$.
Now inductively define the sequence
$\{a_1<a_2<\ldots<a_n<\ldots\}$ as follows:

\smallskip
\begin{itemize}
\item
Pick any $a_1\in A$, and let $B_1=\{x\in\N\mid \{a_1,x\}\in C_i\}$.
Then $\{a_1,\xi\}\in{}^*C_i$ and $\xi\in\hB_1$.

\smallskip
\item
$\xi\in{}^*A\cap{}^* B_1\Rightarrow A\cap B_1$ is infinite.\footnote
{~Here we use the fact that the hyper-extension $\hX$ of a set $X\subseteq\N$
contains infinite numbers if and only if $X$ is infinite.}
Then pick $a_2\in A\cap B_1$ with $a_2>a_1$.

\smallskip
\item
$a_2\in B_1\Rightarrow \{a_1,a_2\}\in C_i$.

\smallskip
\item
$a_2\in A\Rightarrow \{a_2,\xi\}\in{}^*C_i\Rightarrow
\xi\in{}^*\{x\in\N\mid\{a_2,x\}\in{}^*C_1\}={}^*B_2$.

\smallskip
\item
$\xi\in{}^*A\cap{}^* B_1\cap{}^*B_2\Rightarrow$ we can
pick $a_3\in A\cap B_1\cap B_2$ with $a_3>a_2$.

\smallskip
\item
$a_3\in B_1\cap B_2\Rightarrow\{a_1,a_3\},\{a_2,a_3\}\in C_i$,
and so forth.
\end{itemize}

\medskip
\noindent
Then the infinite set $H=\{a_n\mid n\in\N\}$ is such that $[H]^2\subseteq C_i$.
\end{proof}

\medskip
We now give some hints on how iterated hyper-extensions
can be used in partition regularity of equations.
Recall that:

\medskip
\begin{itemize}
\item
An equation $E(X_1,\ldots,X_n)=0$ is [injectively]
\emph{partition regular} over $\N$ if the set of [distinct]
solutions is weakly partition regular on $\N$,
\emph{i.e.}, for every finite coloring
$\N=C_1\cup\ldots\cup C_r$ one finds [distinct] monochromatic
$a_1,\ldots,a_n\in C_i$ such that $E(a_1,\ldots,a_n)=0$.
\end{itemize}

\medskip
A useful nonstandard notion in this context is the following:

\medskip
\noindent
\textbf{Definition.}\
We say that two hyper-natural numbers $\xi,\zeta\in\hN$ are
\emph{indiscernible}, and write $\xi\simeq\zeta$,
if they cannot be distinguished by any hyper-extension, \emph{i.e.}
if for every $A\subseteq\N$ one has
either $\xi,\zeta\in\hA$ or $\xi,\zeta\notin\hA$.\footnote
{~The name ``indiscernible" is borrowed from mathematical
logic. Recall that in model theory two elements are
named \emph{indiscernible} if they cannot be distinguished by any
first-order formula.}

\medskip
Notice that indiscernibility coincides with equality on finite numbers,
because if $k\in\N$ is finite and $\xi\ne k$,
then trivially $k\in\{k\}={}^*\{k\}$ and $\xi\notin{}^*\{k\}$.
Notice also that if $k>1$ is any natural number,
then $k\,\xi\not\simeq \xi$. Indeed, if
$A$ is the set of those natural numbers $n$
with the property that the largest exponent $a$
such that $k^a$ divides $n$ is even, then
$\xi\in\hA\Leftrightarrow k\,\xi\notin\hA$.
A useful property that one can easily prove is the following:
``If $\xi\simeq\zeta$, then for every $f:\N\to\N$ one has
$\hf(\xi)\simeq\hf(\zeta)$."

\medskip
By using the notion of indiscernibility, one can
reformulate in nonstandard terms:

\medskip
\noindent
\textbf{Definition.}\ [Nonstandard]\
An equation $E(X_1,\ldots,X_n)=0$
is [injectively] \emph{partition regular} on $\N$ if
there exist [distinct] hyper-natural numbers
$\xi_1\simeq\ldots\simeq\xi_n$ such that
$E(\xi_1,\ldots,\xi_n)=0$.

\medskip
The following result recently appeared in \cite{cgs}.

\medskip
\begin{theorem}
The equation $X+Y=Z^2$ is \emph{not} partition regular on $\N$,
except for the trivial solution $X=Y=Z=2$.
\end{theorem}

\begin{proof}[Nonstandard proof]
Assume by contradiction that there exist $\alpha\simeq\beta\simeq\gamma$
in $\hN$ such that $\alpha+\beta=\gamma^2$.
Notice that $\alpha,\beta,\gamma$ are infinite, as otherwise
$\alpha=\beta=\gamma=2$ would be the trivial solution.
By the hypothesis of indiscernibility,
$\alpha,\beta,\gamma$ belong to the same
congruence class modulo $5$, say
$\alpha\equiv\beta\equiv\gamma\equiv i\mod 5$ with $0\le i\le 4$.
Now write the numbers in the forms:
$$\alpha=5^a\cdot\alpha_1+i\,;\quad
\beta=5^b\cdot\beta_1+i\,;\quad
\gamma=5^c\cdot\gamma_1+i$$
where $a,b,c>0$ and $\alpha_1,\beta_1,\gamma_1$ are not divisible by $5$.
Pick a function $f:\N\to\N$ such that, for $n\ge 5$, the value
$f(n)$ is the unique $k\not\equiv 0\mod 5$
such that $n=5^hk+i$ for suitable $h>0$ and $0\le i\le 4$.
Observe that $\alpha_1,\beta_1,\gamma_1$ are the images under $\hf$
of $\alpha,\beta,\gamma$ respectively;
so, $\alpha\simeq\beta\simeq\gamma$ implies that
$\alpha_1\simeq\beta_1\simeq\gamma_1$, and therefore
$\alpha_1\equiv\beta_1\equiv\gamma_1\equiv j\not\equiv 0\mod 5$.

\smallskip
The equality $\alpha+\beta=\gamma^2$ implies that
either $i=0$ or $i=2$.
Assume first that $i=0$. In this case
$\gamma^2=5^{2c}\gamma_1^2$ where $\gamma_1^2\equiv j^2\not\equiv 0\mod 5$.
If $a<b$ then $\alpha+\beta=5^a(\alpha_1+5^{b-a}\beta_1)$ where
$\alpha_1+5^{b-a}\beta_1\equiv j\not\equiv 0\mod 5$.
It follows that $2c=a\simeq c$, a contradiction.
If $a>b$ the proof is similar. If $a=b$ then
$\alpha+\beta=5^a(\alpha_1+\beta_1)$
where $\alpha_1+\beta_1\equiv 2j\not\equiv 0\mod 5$,
and also in this case we would have $2c=a\simeq c$, a contradiction.
If $i=2$ then $\gamma^2-4=5^c(5^c\gamma_1^2+4\gamma_1)$
where $5^c\gamma_1^2+4\gamma_1\equiv 4j\not\equiv 0\mod 5$.
Now, in case $a<b$, one has that
$\alpha+\beta-4=5^a(\alpha_1+5^{b-a}\beta_1)$
where $\alpha_1+5^{b-a}\beta_1\equiv j\not\equiv 0\mod 5$,
and so it would follow that
$5^c\gamma_1^2+4\gamma_1=\alpha_1+5^{b-a}\beta_1$.
But then we would have $4j\equiv j$, which is not
possible because $j\not\equiv 0$.
The case $a>b$ is similar. Finally, if $a=b$ then
$\alpha+\beta-4=5^a(\alpha_1+\beta_1)$ where
$\alpha_1+\beta_1\equiv 2j\not\equiv 0\mod 5$,
and it would follow that $4j\equiv 2j$, again
reaching the contradiction $j\equiv 0$.
\end{proof}

\medskip
The notion of indiscernibility naturally extends to the
iterated hyper-extensions of the natural numbers.
\emph{E.g.}, if  $\Omega,\Xi\in\hhN$ then
$\Omega\simeq\Xi$ means that for every $A\subseteq\N$ one has
either $\Omega,\Xi\in{}^{**}\!A$ or $\Omega,\Xi\notin{}^{**}\!A$.
Notice that $\alpha\simeq{}^*\alpha$ for every $\alpha\in\hN$.

\smallskip
In the sequel, a fundamental role will be played by
the following special numbers.

\medskip
\noindent
\textbf{Definition.}\
A hyper-natural number $\xi\in\hN$ is \emph{idempotent}
if $\xi\simeq\xi+{}^*\xi$.\footnote
{~The name ``idempotent" is justified by its
characterization in terms of ultrafilters:
``\emph{$\xi\in\hN$ is idempotent if and only if the corresponding
ultrafilter $\UU_\xi=\{A\subseteq\N\mid \xi\in\hA\}$
is idempotent with respect to the ``pseudo-sum" operation:
$$A\in\U\oplus\V\ \Longleftrightarrow\ \{n\mid A-n\in\V\}\in\U$$
where $A-n=\{m\mid m+n\in A\}$".}
The algebraic structure $(\beta\N,\oplus)$ on the space
of ultrafilters $\beta\N$ and its related generalizations
have been deeply investigated during the
last forty years, revealing a powerful tool for applications
in Ramsey theory and combinatorial number theory
(see the comprehensive monography \cite{hs}).
In this area of research,
idempotent ultrafilters are instrumental.}

\medskip
Recall van der Waerden Theorem:
\emph{``Arbitrarily large monochromatic arithmetic progressions
are found in every finite coloring of $\N$"}.
Here we prove a weakened version about
3-term arithmetic progressions, by showing the
partition regularity of a suitable equation. 

\medskip
\begin{theorem}
The diophantine equation $X_1-2X_2+X_3=0$ is injectively partition regular on $\N$,
which means that for every finite coloring of $\N$
there exists a non-constant monochromatic 3-term arithmetic progression.
\end{theorem}

\begin{proof}[Nonstandard proof]
Pick an idempotent number $\xi\in\hN$.
The following three distinct numbers in ${}^{***}\N$
are a solution of the given equation:
$$\nu=2\xi+\,0\, +{}^{**}\xi\,;\
\mu=2\xi+{}^*\xi+{}^{**}\xi\,;\
\lambda=2\xi+2{}^*\xi+{}^{**}\xi.$$
That $\nu\simeq\mu\simeq\lambda$ are indiscernible
is proved by a direct computation. Precisely,
notice that 
${}^*\xi\simeq\xi+{}^*\xi$ by the idempotency hypothesis,
and so, for every $A\subseteq\N$ and
for every $n\in\N$, we have that
$${}^*\xi\in{}^{**}\!A-n={}^{**}(A-n)\ \Leftrightarrow\ \xi+{}^*\xi\in{}^{**}(A-n).$$
In consequence, the properties listed below are equivalent to each other:

\smallskip
\begin{itemize}
\item
$2\xi+{}^*\xi+{}^{**}\xi\in{}^{***}\!A$
\item
$2\xi\in({}^{***}\!A-{}^{**}\xi-{}^*\xi)\cap\hN={}^*[({}^{**}\!A-{}^*\xi-\xi)\cap\N]$
\item
$2\xi\in{}^*\{n\in\N\mid \xi+{}^*\xi\in{}^{**}(A-n)\}$
\item
$2\xi\in{}^*\{n\in\N\mid {}^*\xi\in{}^{**}(A-n)\}$
\item
$2\xi\in{}^*[({}^{**}\!A-{}^*\xi)\cap\N]=({}^{***}\!A-{}^{**}\xi)\cap\hN$
\item
$2\xi+{}^{**}\xi\in{}^{***}\!A.$
\end{itemize}

\smallskip
This shows that $\nu\simeq\mu$. The other relation
$\mu\simeq\lambda$ is proved in the same fashion.\footnote
{~Here we actually proved the following result
(\cite{bh} Th. 2.10):
\emph{``Let $\U$ be any idempotent ultrafilter.
Then every set $A\in 2\,\U\oplus\U$ contains
a 3-term arithmetic progression"}.}
\end{proof}

\medskip
One can elaborate on the previous nonstandard proof
and generalize the technique. Notice that the considered
elements $\mu,\nu,\lambda$
were linear combinations of iterated hyper-extensions of
a fixed idempotent number $\xi$, and so they can be described
by the corresponding finite strings of coefficients
in the following way:

\medskip
\begin{itemize}
\item
$\nu=2\xi+\,0\, +{}^{**}\xi\ \rightsquigarrow\ \langle 2,0,1\rangle$
\item
$\mu=2\xi+{}^*\xi+{}^{**}\xi\ \rightsquigarrow\ \langle 2,1,1\rangle$
\item
$\lambda=2\xi+2{}^*\xi+{}^{**}\xi\ \rightsquigarrow\ \langle 2,2,1\rangle$
\end{itemize}

\medskip
Indiscernibility of such linear combinations
is characterized by means of a suitable
equivalence relation $\approx$ on the finite strings, so that, \emph{e.g.},
$\langle 2,0,1\rangle\approx\langle 2,1,1\rangle\approx\langle 2,2,1\rangle$.

\medskip
\noindent
\textbf{Definition.}\
The equivalence $\approx$ between (finite)
strings of integers is the smallest equivalence relation such that:

\smallskip
\begin{itemize}
\item
The empty string $\approx\langle 0\rangle$.
\item
$\langle a\rangle\approx\langle a,a\rangle$ for all $a\in\Z$.
\item
$\approx$ is coherent with \emph{concatenations}, \emph{i.e.}
$$\sigma\approx\sigma'\ \text{and }\tau\approx\tau'\ \Longrightarrow\
\sigma^{\frown}\tau\,\approx\,\sigma'^\frown\tau'.$$
\end{itemize}

\medskip
So, $\approx$ is preserved by inserting or removing zeros,
by repeating finitely many times a term or, conversely, by shortening a
block of consecutive equal terms.
The following characterization is proved in \cite{dn2}:

\medskip
\begin{itemize}
\item
Let $\xi\in\hN$ be idempotent.
Then the following are equivalent:

\smallskip
\begin{enumerate}
\item
$a_0\xi+a_1{}^*\xi+\ldots+a_k\cdot{}^{k*}\xi\,\simeq\,
b_0\xi+b_1{}^*\xi+\ldots+b_h\cdot{}^{h*}\xi$

\smallskip
\item
$\langle a_0,a_1,\ldots,a_k\rangle\,\approx\,
\langle b_0,b_1,\ldots,b_h\rangle$.
\end{enumerate}
\end{itemize}

\medskip
Recall Rado theorem: \emph{``The diophantine equation
$c_1 X_1+\ldots+c_n X_n=0$ ($c_i\ne 0)$
is partition regular if and only if $\sum_{i\in F}c_i=0$
for some nonempty $F\subseteq\{1,\ldots,n\}$"}.
By using the above equivalence, one obtains a nonstandard
proof of a modified version of Rado theorem, with
a stronger hypothesis and a stronger thesis.

\medskip
\begin{theorem}
Let $c_1 X_1+\ldots+c_n X_n=0$ be a diophantine
equation with $n\ge 3$. If $c_1+\ldots+c_n=0$
then the equation is injectively partition regular on $\N$.
\end{theorem}

\begin{proof}[Nonstandard proof]
Fix $\xi\in\hN$ an idempotent element, and for simplicity
denote by $\xi_i={}^{i*}\xi$ the $i$-th iterated hyper-extension of $\xi$.
For arbitrary $a_1,\ldots,a_{n-1}$, consider
the following numbers in ${}^{n*}\N$:
$$\!\!\!\!\!
\tiny{
\begin{array}{lllllllllllllll}
\mu_1 &=& a_1\xi &+& a_2\xi_1 &+& a_3\xi_2 &+& \ldots &+& a_{n-2}\xi_{n-3} &+& a_{n-1}\xi_{n-2} &+& a_{n-1}\xi_{n-1} \\
\mu_2 &=& a_1\xi &+& a_2\xi_1 &+& a_3\xi_2 &+& \ldots &+& a_{n-2}\xi_{n-3} &+& 0 &+& a_{n-1}\xi_{n-1} \\
\mu_3 &=& a_1\xi &+& a_2\xi_1 &+& a_3\xi_2 &+& \ldots &+& 0 &+& a_{n-2}\xi_{n-2} &+& a_{n-1}\xi_{n-1} \\
\vdots &{}& \vdots &{}& \vdots &{}& \vdots &{}& \vdots &{}& \vdots &{}& \vdots &{}& \vdots \\
\mu_{n-2} &=& a_1\xi &+& a_2\xi_1 &+& 0 &+& a_3\xi_3 &+& \ldots &+& a_{n-2}\xi_{n-2} &+& a_{n-1}\xi_{n-1} \\
\mu_{n-1} &=& a_1\xi &+& 0 &+& a_2\xi_2 &+& a_3\xi_3 &+& \ldots &+& a_{n-2}\xi_{n-2} &+& a_{n-1}\xi_{n-1} \\
\mu_n &=& a_1\xi &+& a_1\xi_1 &+& a_2\xi_2 &+& a_3\xi_3 &+& \ldots &+& a_{n-2}\xi_{n-2} &+& a_{n-1}\xi_{n-1}
\end{array}}
$$

\medskip
Notice that $\mu_1\simeq\ldots\simeq\mu_n$ because the corresponding
strings of coefficients are all equivalent to
$\langle a_1,\ldots,a_{n-1}\rangle$.
Moreover, it can be easily checked that the $\mu_i$s are distinct.
To complete the proof, we need to find
suitable coefficients $a_1,\ldots,a_{n-1}$
in such a way that $c_1\mu_1+\ldots+c_n\mu_n=0$.
It is readily seen that this happens if the following conditions are fulfilled:

$$\begin{cases}
(c_1+\ldots+c_n)\cdot a_1 = 0 \\
(c_1+\ldots+c_{n-2})\cdot a_2 + c_n\cdot a_1 = 0 \\
(c_1+\ldots+c_{n-3})\cdot a_3+ (c_{n-1}+c_n)\cdot a_2 = 0 \\
{}\quad\quad\quad \vdots \\
c_1\cdot a_{n-1}+(c_3+\ldots+c_n)\cdot a_{n-2} = 0 \\
(c_1+\ldots+c_n)\cdot a_{n-1} = 0
\end{cases}$$

\medskip
Finally, observe that the first and last equations are trivially
satisfied because of the hypothesis $c_1+\ldots+c_n=0$;
and the remaining $n-2$ equations are satisfied by infinitely many
choices of the coefficients $a_1,\ldots,a_{n-1}$, which can
be taken in $\N$.\footnote
{~Here we actually proved the following result (\cite{dn2} Th.1.2):
\emph{``Let $c_1 X_1+\ldots + c_n X_n=0$ be a diophantine
equation with $c_1+\ldots+c_n=0$ and $n\ge 3$.
Then there exists $a_1,\ldots,a_{n-1}\in\N$
such that for every idempotent ultrafilter $\U$
and for every $A\in a_1\U\oplus \ldots\oplus a_{n-1}\U$
there exist distinct $x_i\in A$ such that
$c_1 x_1+\ldots+c_n x_n=0$"}.}
\end{proof}

\medskip
More results in this direction, including partition regularity
of non-linear diophantine equations, have been recently obtained
by L.~Luperi Baglini (see \cite{lu}).

\bigskip
\section{A model of the hyper-integers}

In this final section we outline a construction for a model
where one can give an interpretation
to all nonstandard notions and principles that were considered
in this paper.

\smallskip
The most used single construction for models of
the hyper-real numbers, and hence of the hyper-natural and
hyper-integer numbers, is the \emph{ultrapower}.\footnote
{~For a comprehensive exposition of nonstandard analysis
grounded on the ultrapower construction, see
R. Goldblatt's textbook \cite{go}.}
Here we prefer to use the purely algebraic construction of \cite{bd},
which is basically equivalent to an ultrapower, but where
only the notion of quotient field of a ring
modulo a maximal ideal is assumed.

\medskip
\begin{itemize}
\item
Consider $\text{Fun}(\N,\R)$, the
ring of real sequences $\varphi:\N\to\R$
where the sum and product operations are defined pointwise.

\smallskip
\item
Let $\mathfrak{I}$ be the ideal of the sequences that
eventually vanish:
$$\mathfrak{I}\ =\ \{\varphi\in\text{Fun}(\N,\R)\mid
\exists k\,\forall n\ge k\ \varphi(n)=0\}.$$

\smallskip
\item
Pick a maximal ideal $\mathfrak{M}$ extending $\mathfrak{I}$,
and define the hyper-real numbers as the quotient field:
$$\hR\ =\ {\text{Fun}(\N,\R)}/\mathfrak{M}.$$

\smallskip
\item
The \emph{hyper-integers} are the subring of $\hR$ determined by the
sequences that take values in $\Z$:
$$\hZ\ =\ {\text{Fun}(\N,\Z)}/\mathfrak{M}\ \subset\ \hR.$$

\smallskip
\item
For every subset $A\subset\R$, its hyper-extension is defined by:
$$\hA\ =\ {\text{Fun}(\N,A)}/\mathfrak{M}\ \subset\ \hR.$$
So, \emph{e.g.}, the \emph{hyper-natural numbers} $\hN$
are the cosets $\varphi+\mathfrak{M}$ of
sequences $\varphi:\N\to\N$ of natural numbers; the hyper-prime numbers
are the cosets of sequences of prime numbers, and so forth.

\smallskip
\item
For every function $f:A\to B$ (where $A,B\subseteq\R$), its
hyper-extension $\hf:\hA\to\hB$ is defined by putting
for every $\varphi:\N\to A$:
$$\hf(\varphi+\mathfrak{M})\ =\ (f\circ\varphi)+\mathfrak{M}.$$

\smallskip
\item
For every sequence $\langle A_n\mid n\in\N\rangle$ of nonempty subsets
of $\R$, its hyper-extension $\langle A_\nu\mid\nu\in\hN\rangle$
is defined by putting for every $\nu=\varphi+\mathfrak{M}\in\hN$:
$$A_\nu\ =\ \{\psi+\mathfrak{M}\mid \psi(n)\in A_{\varphi(n)}\
\text{for all }n\}\ \subseteq\ \hR.$$
\end{itemize}

\medskip
It can be directly verified that $\hR$ is an ordered field whose
positive elements are $\hR^+=\text{Fun}(\N,\R^+)/\mathfrak{M}$.
By identifying each $r\in\R$ with the coset $c_r+\mathfrak{M}$
of the corresponding constant sequence, one obtains that $\hR$ is a
proper superfield of $\R$.
The subset $\hZ$ defined as above is a discretely ordered ring
having all the desired properties.

\smallskip
Remark that in the above model, one can interpret
all notions used in this paper.
We itemize below the most relevant ones.

\smallskip
Denote by $\alpha=\imath+\mathfrak{M}\in\hN$ the infinite
hyper-natural number corresponding to the identity sequence $\imath:\N\to\N$.

\medskip
\begin{itemize}
\item
The nonempty \emph{internal sets} $B\subseteq\hR$ are the sets of the form $B=A_\alpha$
where $\langle A_n\mid n\in\N\rangle$ is a sequence of nonempty sets.
When all $A_n$ are finite, $B=A_\alpha$ is called \emph{hyper-finite};
and when all $A_n$ are infinite, $B=A_\alpha$ is called \emph{hyper-infinite}.\footnote
{~It is proved that any internal set $A\subseteq\hR$ is either
hyper-finite or hyper-infinite.}

\smallskip
\item
If $B=A_\alpha$ is the hyper-finite set corresponding
to the sequence of nonempty finite sets $\langle A_n\mid n\in\N\rangle$,
then its \emph{internal cardinality} is defined by setting
$\|B\|=\vartheta+\mathfrak{M}\in\hN$
where $\vartheta(n)=|A_n|\in\N$ is the sequence of cardinalities.

\smallskip
\item
If $\varphi,\psi:\N\to\Z$ and
the corresponding hyper-integers
$\nu=\varphi+\mathfrak{M}$ and $\mu=\psi+\mathfrak{M}$
are such that $\nu<\mu$,
then the (internal) interval $[\nu,\mu]\subseteq\hZ$ is defined as $A_\alpha$
where $\langle A_n\mid n\in\N\rangle$ is any sequence of sets
such that $A_n=[\varphi(n),\psi(n)]$ whenever $\varphi(n)<\psi(n)$.\footnote
{~One can prove that this definition is well-posed. Indeed,
if $\varphi+\mathfrak{M}<\psi+\mathfrak{M}$ and
$\langle A_n\mid n\in\N\rangle$ and $\langle A'_n\mid n\in\N\rangle$
are two sequences of nonempty sets such that $A_n=A'_n$ whenever
$\varphi(n)<\psi(n)$, then $A_\alpha=A'_\alpha$.}
\end{itemize}

\medskip
In full generality, one can show that the \emph{transfer}
principle holds. To show this in a rigorous manner, one needs
first a precise definition of ``elementary property", which requires
the formalism of first-order logic. Then, by using a procedure
known in logic as ``induction on the complexity of formulas",
one proves that the equivalences
$P(A_1,\ldots,A_k,f_1,\ldots,f_h)\Leftrightarrow
P(\hA_1,\ldots,\hA_k,{}^*f_1,\ldots,{}^*f_h)$
hold for all elementary properties $P$, sets $A_i$,
and functions $f_j$.

\smallskip
Remark that all the nonstandard definitions given in this paper are actually
equivalent to the usual ``standard" ones. As examples,
let us prove some of those equivalences in detail.

\smallskip
Let us start with the definition of a \emph{thick set} $A\subseteq\Z$.
Assume first that there exists a sequence
of intervals $\langle\, [a_n,a_n+n]\mid n\in\N\,\rangle$ which
are included in $A$. If $\langle\, [a_\nu,a_\nu+\nu]\mid \nu\in\hN\,\rangle$
is its hyper-extension then, by \emph{transfer}, every
$[a_\nu,a_\nu+\nu]\subseteq\hA$, and hence $\hA$ includes infinite intervals.
Conversely, assume that $A$ is not thick and pick $k\in\N$ such
that for every $x\in\Z$ the interval $[x,x+k]\nsubseteq A$. Then, by \emph{transfer},
for every $\xi\in\hZ$ the interval $[\xi,\xi+k]\nsubseteq\hA$, and hence
$\hA$ does not contain any infinite interval.

\smallskip
We now focus on the nonstandard definition of \emph{upper Banach density}.
Let $\text{BD}(A)\ge\beta$. Then for every $k\in\N$, there exists an
interval $I_k\subset\Z$ of length $|I_k|\ge k$ and such that
$|A\cap I_k|/|I_k|>\beta-1/k$.
By \emph{overflow}, there exists an infinite $\nu\in\hN$
and an interval $I\subset\hZ$ of internal cardinality $\|I\|\ge\nu$
such that the ratio $\|\hA\cap I\|/\|I\|\ge\beta-1/\nu\sim\beta$.
Conversely, let $I$ be an infinite interval such that
$\|\hA\cap I\|/\|I\|\sim\beta$. Then, for every given $k\in\N$,
the following property holds:
``There exists an interval $I\subset\hZ$ of length $\|I\|\ge k$
and such that $\|\hA\cap I\|/\|I\|\ge\beta-1/k$". By
\emph{transfer}, we obtain the existence of an interval
$I_k\subset\Z$ of length $|I_k|\ge k$ and such that
$|A\cap I_k|/|I_k|\ge\beta-1/k$. This shows that $\text{BD}(A)\ge\beta$,
and the proof is complete.

\smallskip
Let us now turn to \emph{finite embeddability}.
Assume that $X\fe Y$, and enumerate $X=\{x_n\mid n\in\N\}$.
By the hypothesis, $\bigcap_{i=1}^n(Y-x_i)\ne\emptyset$ for every $n\in\N$ and so,
by \emph{overflow}, there exists an infinite $\mu\in\hN$
such that the hyper-finite intersection
$\bigcap_{i=1}^\mu(\hY-x_i)\ne\emptyset$.
If $\nu$ is any hyper-integer in that intersection,
then $\nu+X\subseteq\hY$.
Conversely, let us assume that $\nu+X\subseteq\hY$ for a
suitable $\nu\in\hZ$.
Then for every finite $F=\{x_1,\ldots,x_k\}\subset X$
one has the elementary property:
``$\exists \nu\in\hZ\ (\nu+x_1\in\hY\ \&\ \ldots\ \&\ \nu+x_k\in\hY)$".
By \emph{transfer}, it follows that
``$\exists t\in\Z\ (t+x_1\in Y\ \&\ \ldots\ \&\ t+x_k\in Y)$",
\emph{i.e.} $t+F\subseteq Y$.\footnote{
~For the equivalence of the nonstandard definition
of partition regularity of an equation, one
needs a richer model than the one presented here.
Precisely, one needs the so-called
$\mathfrak{c}^+$-\emph{enlargement property},
that can be obtained in models of the form
$\hR={\text{Fun}(\R,\R)}/\mathfrak{M}$ where $\mathfrak{M}$
is a maximal ideals of a special kind (see \cite{bd}).}

\smallskip
We finish this paper with a few suggestions for further readings.
A rigorous formulation and a detailed proof of the
\emph{transfer principle} can be found in Ch.4 of the
textbook \cite{go}, where the \emph{ultrapower} model is considered.\footnote{
~Remark that our algebraic model is basically equivalent to
an ultrapower. Indeed, for any maximal ideal
$\mathfrak{M}$ of the ring $\text{Fun}(\N,\R)$, the family
$\U=\{Z(\varphi)\mid \varphi\in\mathfrak{M}\}$ where
$Z(f)=\{n\in\N\mid\varphi(n)=0\}$ is an ultrafilter on $\N$.
By identifying each coset $\varphi+\mathfrak{M}$ with the corresponding
$\U$-equivalence class $[\varphi]$, one obtains
that the quotient field ${\text{Fun}(\N,\R)}/\mathfrak{M}$
and the ultrapower $\R^\N/\U$ are essentially the same object.}
See also \S 4.4 of \cite{ck} for the foundations of
nonstandard analysis in its full generality.
A nice introduction of nonstandard methods
for number theorists, including a number of examples,
is given in \cite{jincant} (see also \cite{jin0}).
Finally, a full development of nonstandard analysis can be
found in several monographies of the existing literature;
see \emph{e.g.} the classical H.J. Keisler's book \cite{kei},
or the comprehensive collections of surveys in \cite{nato}.

\end{document}